\documentclass{amsart}

\usepackage{amssymb}
\usepackage{amsmath}
\usepackage{amscd}
\usepackage{latexsym}
\usepackage{bbold}

\def\m{\bf m}
\def\n{\bf n}

\def \det{\operatorname{det}}
\def \adj{\operatorname{adj}}

\def\m{\bold m}

\def \ord{\operatorname{ord}}
\def \Sym{\operatorname{Sym}}
\def \rk{\operatorname{rk}}
\def \Hom{\operatorname{Hom}_R}

\newtheorem{theorem}{Theorem}[section]
\newtheorem{lemma}[theorem]{Lemma}
\newtheorem{corollary}[theorem]{Corollary}
\newtheorem{proposition}[theorem]{Proposition}

\theoremstyle{definition}
\newtheorem{definition}[theorem]{Definition}

\theoremstyle{remark}
\newtheorem{remark}[theorem]{Remark}

\numberwithin{equation}{section}

\begin{document}

\title[Lengths and multiplicities of integrally closed modules]{Lengths and multiplicities of integrally closed modules over a two-dimensional regular local ring}

\author{Vijay Kodiyalam}
\address{The Institute of Mathematical Sciences, Taramani, Chennai 600 113, India}
\email{vijay@imsc.res.in}

\author{Radha Mohan}
\address{St. Stephen's College, Delhi 110 007, India}
\email{radhamohan@vsnl.com}
\thanks{The second named author was supported by the Institute of Mathematical 
Sciences during her visits there.}

\subjclass[2010]{Primary 13B21, 13B22, 13H05, 13H15}

\date{}

\dedicatory{Dedicated to Bill Heinzer on his seventy fifth birthday}

\begin{abstract}

Let $(R,{\bf m})$ be a two-dimensional regular local ring with infinite residue 
field. We prove an analogue of the Hoskin-Deligne length formula for a finitely generated,
torsion-free, integrally closed $R$-module $M$. As a consequence, we get
a formula for the Buchsbaum-Rim multiplicity of $F/M$, where 
$F = M^{**}$.
\end{abstract}

\maketitle

\section{Introduction}

The theory of integrally closed or complete ideals in a two-dimensional
regular local ring was founded by Zariski in \cite{Zrs1938}. Since,
then this theory has received a good deal of attention and has been refined
and generalized. The first named author generalized this theory to finitely
generated, torsion-free, integrally closed modules in \cite{Kdy1993, Kdy1995}. While the structural cornerstones
of this theory are Zariski's product  and unique factorization
theorems, the basic numerical result here is the Hoskin-Deligne
length formula.

This formula has several proofs beginning with the one by Hoskin in \cite{Hsk1956}, 
 through proofs by Deligne in \cite{Dlg1973}, by Rees in \cite{Res1981}, by Lipman in \cite{Lpm1987}, by Johnston and Verma in
\cite{JhnVrm1992}, and
 by the first named author in \cite{Kdy1993} (which is based on techniques of \cite{Lpm1987} and \cite{Hnk1989}), to the one in \cite{DbrLrd2002}.

In this paper we obtain an analogue of the Hoskin-Deligne formula
for finitely generated, torsion-free, integrally closed modules over a 
two-dimensional regular local ring.
A consequence of this is a formula for the Buchsbaum-Rim multiplicity
of a certain finite length module associated to an integrally closed module.

We now summarise the rest of the paper. In \S 2, we collect various facts and results about
integrally closed modules and reductions from \cite{Res1987}, about integrally closed
modules and their transforms over two-dimensional regular local rings from \cite{Kdy1995} and
about Buchsbaum-Rim multiplicities from \cite{BchRim1964}.
In \S3, we prove the
analogue of the Hoskin-Deligne formula for integrally closed modules which expresses their
colength in terms of those of modules contracted from the order valuation rings of various
quadratic transforms of the base ring. In the final \S4 we  apply our analogue of the Hoskin-Deligne formula to prove a Buchsbaum-Rim
multiplicity formula for such modules.

\section{Preliminaries}

\subsection{Integral Closures and Reductions of Modules}

We review the notions of integral closures and reductions for torsion-free
modules over arbitrary Noetherian domains as developed by Rees \cite{Res1987}.

Throughout this subsection, $R$ will be a Noetherian domain with field of fractions $K$ and $M$
will be a finitely generated, torsion-free $R$-module. 
We denote its rank by $rk_R(M)$.
By $M_K$ we denote
the  $rk_R(M)$-dimensional $K$-vector space $M \otimes_R K$. If $N$ is a
submodule of $M$ then $N_K$ is naturally identified with a subspace of $M_K$.

Any ring between $R$  and $K$ is said to be a birational
overring of $R$. For any such birational overring $S$ of $R$, we let $MS$ denote
the $S$-submodule of $M_K$ generated by $M$. There is a canonical $R$-module
homomorphism from $M \otimes_R S$ onto $MS$ with kernel being the submodule
of $S$-torsion (equivalently $R$-torsion) elements. Hence, $M \otimes_R S$ modulo $S$-torsion and
$MS$ are isomorphic as $S$-modules.

Let $S(M)$ denote the image of the symmetric algebra ${\Sym}^R (M)$ in the 
algebra ${\Sym}^K (M_K)$ under the canonical map. As an $R$-algebra $S(M)$
is ${\Sym}^R (M)$ modulo its ideal of $R$-torsion elements. Let $S_n(M)$ 
denoted the $n$-th graded component of the positively graded $R$-algebra
$S(M)$.

Similarly, let $E(M)$ denote the image of the exterior algebra $\wedge^R (M)$
in the algebra $\wedge^K (M_K)$ under the canonical map. As an $R$-algebra $E(M)$
is $\wedge^R (M)$ modulo its ideal of $R$-torsion elements. Let $E_n(M)$ 
denoted the $n$-th graded component of the positively graded $R$-algebra
$E(M)$.  Observe that if $\rk (M) = r$ then ${\wedge}_r (M)$ is an $R$-module
of rank $1$ contained in $K$ and hence is isomorphic to an ideal of $R$.
If $N$ is a submodule of $M$ of the same rank, then $E_r(N)$ is contained
in $E_r(M)$ and fixing an isomorphism of $E_r(M)$ with an ideal of $R$,
we can identify $E_r(N)$ as a subideal.

The following fundamental result of Rees - see Theorems 1.2 and 1.5 of \cite{Res1987} - is the basic theorem in integral closures and reductions
of modules.

\begin{theorem}\label{basic} Let $R$ be a Noetherian domain with field of fractions $K$ and let $M$
be a finitely generated, torsion-free $R$-module of rank $r$. For an
element $v \in M_K$, the following conditions are equivalent:
\begin{itemize}
\item{Valuative criterion:} $v \in MV$ for every discrete valuation ring $V$ of
$K$ containing $R$.
\item{Equational criterion:} The element $v \in M_K = {\Sym}_1^K (M_K)$ is integral over $S(M)$.
\item{Determinantal criterion:} Under some (any)  isomorphism of $E_r(M + Rv)$ with an ideal $I$ of $R$,
the subideal $J$ corresponding to $E_r(M)$ is a reduction of $I$.
\end{itemize}
\end{theorem}

\begin{definition} With notation as above, the element $v \in M_K$
is said to be integral over $M$ if the equivalent conditions of Theorem \ref{basic} hold. The integral closure of $M$, denoted $\overline{M}$, is
the set of all elements of $M_K$ that are integral over $M$. The module
$M$ is integrally closed if $\overline{M} = M$.
A submodule $N$ is a reduction of $M$ if
$M \subset \overline{N}$.
Over a local domain, a minimal reduction is one for which the minimal number of generators
is minimal among all reductions.
\end{definition}

It is clear that $M \subseteq \overline{M} \subseteq M_K$. In addition,
Rees shows in \cite{Res1987} that if $R$ is a Noetherian normal domain then
there are natural inclusions
$\overline{M} \subseteq M^{**} \subseteq M_K$ where $(-)^{*}$ denotes the
functor $\Hom (-,R)$.

The following result of Rees - see Lemma 2.1 of \cite{Res1987} - generalizes to modules the
theorem that any ${\bf m}$-primary ideal of a $d$-dimensional, Noetherian local
ring $(R, {\bf m})$ with infinite residue field has a $d$-generated reduction
where $d > 0$.

\begin{theorem} Let $R$ be a $d$-dimensional, Noetherian local
domain with infinite residue field and $M$ be a non-free, finitely generated,
torsion-free $R$-module. Then $M$ has a minimal reduction which is generated
by at most $\rk (M) + d -1$ elements. Further, a minimal generating set of a minimal
reduction of $M$ forms part of a minimal generating set for $M$. In particular, when
$d = 2$, $M$ has a $\rk (M) + 1$ generated minimal reduction.
\end{theorem}

\subsection{Contracted Modules and Module Transforms}\label{Kdysum}

In this subsection, we summarise the results of \cite{Kdy1995} that we will use in the sequel.
Throughout this subsection, $R$ will be a two-dimensional regular local ring with maximal
ideal ${\bf m}$, infinite
residue field $k$ and field of fractions $K$.
The discrete valuation determined by the powers of ${\bf m}$ is denoted $ord_R(\cdot)$
and the associated valuation ring is denoted by $V_R$ or simply by $V$.

Throughout, $M$ be a finitely generated, torsion-free $R$-module and the notations
$\lambda_R(\cdot)$ and $\nu_R(\cdot)$ will denote respectively the length and minimal number of
generators functions on $R$-modules.
We will reserve $F$ to stand for 
the double dual $M^{**}$. It is a fact that $F$ is free (of rank $rk_R(M)$) and canonically
contains $M$ with quotient of finite length. Further, these properties characterise $M^{**}$ upto unique
isomorphism (restricting to the identity on $M$).

Let $G$ be any free module containing $M$ and of the same rank as $M$.
Choose a basis for $G$ and a minimal generating set for $M$ and
consider the matrix expressing this set of generators in terms of the chosen
basis of $G$. Considering the elements of $G$ as column vectors we get a
$\rk_R(M) \times \nu_R(M)$ representing matrix for $M$.
The ideal of
maximal minors, i.e., $\rk (M)$-sized minors, of this matrix is denoted $I_G(M)$.
It is easy to see that $I_G(M)$ is independent of the choices made and is
an invariant of the imbedding $M \subseteq G$. 
If $G=F=M^{**}$, then we will write $I(M)$ for $I_G(M)$.
If $M$ is a free module, $I(M) = R$ and if
$M$ is non-free then it follows from the 
fact that $F/M$ is of finite length that $I(M)$ is ${\bf m}$-primary.
We define the order of $M$ denoted by $ord_R(M)$ to be $ord_R(I(M))$.

We will also have occasion to use the following simple lemma.

\begin{lemma}\label{prinfree}
Let $R$ be a  Noetherian local domain, $G$ be a free module of finite rank
and $M \subseteq G$ a submodule of rank equal to $rk_R(G)$. Then $M$ is free
iff $I_G(M)$ is principal.
\end{lemma}

\begin{proof} One implication being trivial, we prove the other. Suppose that
$I_G(M)$ is principal. Since $R$ is local, some maximal minor of a
representing matrix of $M$ generates $I_G(M)$.
Say it corresponds to some $rk_R(G)$ columns of such a matrix. These generate some free
submodule of $M$ of rank $rk_R(G)$. We claim that this submodule is $M$ itself.
For take any other column of $M$. Write it as a linear combination of the chosen
$rk_R(G)$ columns with coefficients in $K$ - the field of fractions of $R$.
Consider the minor of $M$ obtained by replacing one of the $rk_R(G)$ columns by
this column. On the one hand this gives the corresponding coefficient times the
generating maximal minor of $I_G(M)$. On the other hand this lies in $I_G(M)$. It follows that the coefficient is in $R$. So this column
lies in the $R$-submodule generated by the chosen $rk_R(G)$ columns .
\end{proof}

\begin{definition} Let $M$ be a finitely generated, torsion-free $R$-module.
Let $S$ be a birational overring of $R$ of the form $R[\frac{\bf m}{x}]$ where $x$ is
a minimal generator of ${\bf m}$. We call $MS$ the transform
of $M$ in $S$. The module $M$ is said to be contracted from $S$ if $M = MS \cap F$ 
regarded as submodules of $FS$.
\end{definition}

The following proposition - see Proposition 2.5 of \cite{Kdy1995} - is a useful characterisation of contracted modules.

\begin{proposition}\label{contracted}
With notation as above $\nu_R(M) \leq ord_R(M)+rk_R(M)$. Further,
the following conditions are equivalent:
\begin{enumerate}
\item There exists $x \in \m \backslash \m^2$ such that $M$ is contracted from $S=R[\frac{\bf m}{x}]$.
\item There exists $x \in \m \backslash \m^2$ such that $(M:_Fx) = (M:_F \m)$.
\item There exists $x \in \m \backslash \m^2$ such that $\lambda_R(F/(xF+M)) = \nu_R(M) - rk_R(M)$.
\item $ord_R(M) = \nu_R(M) - rk_R(M)$.
\item For any $x \in \m \backslash \m^2$ such that $ord_R(M) = \lambda_R(R/(x,I(M)))$, $M$ is contracted
from $S=R[\frac{\bf m}{x}]$.
\end{enumerate}
\end{proposition}

A first quadratic transform of $R$ is a ring obtained by localizing a ring of
the form $S = R[\frac{\bf m}{x}]$ (as above)  at a maximal
ideal containing ${\bf m}S$. Such a ring is itself a two-dimensional regular local
ring and we define an $n$-th quadratic transform of $R$ as a
first quadratic transform of an $(n-1)$st quadratic transform of $R$. In
general, a quadratic transform of $R$ is a $n$-th quadratic transform
of $R$ for some $n$. By convention, we regard $R$ itself as a quadratic transform
of $R$ with $n = 0$. 

By well-known results on quadratic transforms - see p392 of \cite{ZrsSml1960} - if $T$ is 
a quadratic transform of $R$, there is a unique sequence of quadratic 
transforms, $R = T_0 \subset T_1 \subset \hdots \subset T_n = T$,
where each $T_{i+1}$ is a first quadratic transform of $T_i$ for
$i = 0, \hdots , n-1$.
Further, if $\m = (x,y)R$, then any first quadratic transform of $R$ is either a 
localisation of $R[\frac{\m}{x}]$ or is the localisation of $R[\frac{\m}{y}]$ at the maximal ideal $(y,\frac{x}{y})R[\frac{\m}{y}]$.

For a finitely generated, torsion-free $R$-module $M$ and a quadratic
transform $T$ of $R$, the transform of $M$ in $T$ is defined to be the module $MT$.
For an ideal $I$, we consider also the related notion of proper transform, denoted  $I^T$, which is defined to be the $\m_T$-primary ideal
$x^{-ord_T(IT)}IT$, where $x \in \m_RT$ is a generator.

We will use the following result - see Proposition 4.3, Proposition 4.6, Theorem 5.2, Theorem 5.3, Theorem 5.4 of \cite{Kdy1995}.

\begin{theorem}\label{Kdymain}
If $M$ is an integrally closed module over a two-dimensional regular local ring $R$,
then, for most $x \in \m \backslash \m^2$, $M$ is contracted from $S=R[\frac{\bf m}{x}]$.
The ideal $I(M)$ is integrally closed.
All the symmetric powers of $M$ modulo $R$-torsion, $S_n(M)$, are integrally closed.
The transform $MT$ of $M$ in a quadratic transform $T$ of $R$ is integrally closed.
If $M$ and $N$ are both integrally closed, so is $MN$.
\end{theorem}

\subsection{Buchsbaum-Rim Multiplicity}

Let $R$ be a Noetherian local ring of dimension $d$. Let $P$ be an $R$-module
of finite length with a free presentation
$$
G\rightarrow F \rightarrow P \rightarrow 0.
$$
Buchsbaum and Rim - see Theorem 3.1 of \cite{BchRim1964} - showed that if $R$ is a Noetherian
local ring of dimension $d$ and $P$ is a finite length, non-zero
$R$-module and $S(G)$ is the image of $\Sym^R(G)$ in $\Sym^R (F)$, then $\lambda_R({\Sym}_n^R(F)/S(G)$ is asymptotically given by
a polynomial function, $p(n)$, of $n$ of degree $\rk (F) + d - 1$ and that the
normalized leading coefficient is independent of the presentation chosen.

\begin{definition} With notation as above, the normalized leading
coefficient of $p(n)$, is an invariant of $P$ and is
called the Buchsbaum-Rim multiplicity of $P$. The Buchsbaum-Rim multiplicity
of the zero module is defined to be zero.
\end{definition}

In a two-dimensional regular local ring $(R,{\bf m},k)$ with infinite residue
field and a finitely generated, torsion-free $R$-module $M$, we let
$e_R(M)$ denote the Buchsbaum-Rim multiplicity of $F/M$ where 
$F = M^{**}$.
We will need the following result that is a consequence of Corollary 4.5 of \cite{BchRim1964} and Proposition 3.8 of \cite{Kdy1995}.

\begin{proposition}\label{amalgam}
Let $(R,{\bf m},k)$ be a two-dimensional regular local ring with infinite residue field and $M \subseteq F = M^{**}$ a finitely generated, torsion-free $R$-module with minimal reduction $N$. Then
$$
e_R(M) = e_R(N) = \lambda_R(F/N)
$$
\end{proposition}

\section{Analogue of the Hoskin-Deligne length formula}\label{lmf}

All notation in this section will be as in \S \ref{Kdysum}. In particular, $R$ will be
a two-dimensional regular local ring with maximal ideal $\m$, infinite
residue field $k$ and field of fractions $K$ and $M$ will be a finitely generated,
torsion-free $R$-module with double dual $F$. The order valuation ring of $R$ will
be denoted by $V$. These notations will be in force in the statements of all results
of this section.

\subsection{On modules contracted from the order valuation ring}
The goal of this subsection is to study
some properties of modules contracted from
the order valuation ring $V$ of $R$. These
will form the basic building blocks in the
analogue of the Hoskin-Deligne formula.

We begin with the following lemma which will be used to compute the contraction
of a module extended to $V$.

\begin{lemma}\label{Clem} There exists a free submodule $C \subseteq M$
of rank equal to $rk_R(M)$ such that $CV \cap F = MV \cap F$.
Further, $ord_R(det(C)) = ord_R(M)$.
\end{lemma}
\begin{proof} Consider a matrix representation of $M \subseteq F$ as a $rk_R(M) \times \nu_R(M)$
matrix over $R$ so that $M$ is generated by the columns of this matrix. Hence so is
$MV \subseteq FV$, as a module over $V$. Since $V$ is a principal ideal domain and
$MV$ and $FV$ are of equal rank, some $rk_R(M)$ columns of this matrix generate
$MV$. Let $C$ be the $R$-submodule of $M$ generated by these columns. Then $C$ is free
of rank $rk_R(M)$ and $CV = MV$ so that $CV \cap F = MV \cap F$. Appeal to Lemma \ref{ordlemma} to see that $ord_R(det(C)) = ord_R(M)$.
\end{proof}

\begin{lemma}\label{compute} Suppose that $C \subseteq F$ are free modules of equal rank
and let $n = ord_R(det(C))$.
Then $CV \cap F =(\m^nC :_F \det (C)).$
\end{lemma}
\begin{proof} Let $ w \in (\m^nC :_F \det (C))$. Then
$w\det (C) \in \m^nC \subseteq \m^nCV$. Since $det(C)$
generates $\m^nV$, it follows that $w \in CV \cap F$
showing that $(\m^nC:_F \det (C)) \subseteq CV \cap F$.

To see the opposite inclusion, observe that
$\det (C)(CV \cap F) = \det (C)CV \cap \det (C)F$.
Since, $\det (C) = C \adj(C)$ we have that
$ \det (C)CV \cap \det (C)F \subseteq \m^nCV \cap C = \m^nC$.
This proves that 
 $CV \cap F \subseteq (\m^nC :_F \det (C))$.
\end{proof}

\begin{proposition}\label{algclosed}
Suppose that $M = MV \cap F$, $ord_R(M) = n$ and that
 the residue field $k$ of $R$ is algebraically closed. Then,
$I(M) = \m^n$.
\end{proposition}
\begin{proof} 
Choose $C \subseteq F$ as in Lemma \ref{Clem}, so that $ord_R(det(C)) = n$
and set $I = I(M)$ so that $det(C) \in I \subseteq \m^n$. Since $I$ is $m$-primary,
there is a smallest $t$ so that $\m^t \subseteq I$. We will show that $t > n$ leads
to a contradiction.

Observe that since $M = MV \cap F$ and all $S = R[\frac{\m}{x}]$ are contained
in $V$, we also have $M = MS \cap F$ for all $x \in \m \backslash \m^2$. It then
follows from Proposition \ref{contracted} that $I$ is also contracted from all
$S = R[\frac{\m}{x}]$. Thus, again by Proposition \ref{contracted}, $I:x = I:\m$ for
every minimal generator $x$ of $\m$.

If now, $t > n$, there exists $z \in I$ of order $t-1$. 
For instance $z$ could be chosen to be an appropriate multiple of $det(C)$.
Since $k$ is assumed to
be algebraically closed the image of $z$ in the graded ring $gr_\m(R) \cong k[X,Y]$
is a product of $t-1$ linear factors. Lifting back to $R$ shows that $z-z_1z_2\cdots z_{t-1}
\in \m^t$ where each $z_k$ is a minimal generator of $\m$. Thus $z_1\cdots z_{t-1} \in I$. Since each $I : z_k = I : \m$, it follows that $\m^{t-1} \subseteq I$, contradicting
choice of $t$.

Thus $t \leq n$ and so $I = \m^n$.
\end{proof}

Recall - see Proposition 6.8.2 of Chapter 0 of \cite{GrtDdn1971} - that 
a local ring $(R,\m)$ admits a faithfully flat local overring $(\tilde{R},\tilde{\m})$
such that $\m \tilde{R} = \tilde{\m}$ and such that $\tilde{R}$ has algebraically closed residue
field. If $R$ is a regular local ring, then the dimension formula - see
Theorem 15.1 of \cite{Mts1986} - implies that so is $\tilde{R}$. We will use this to drop that
requirement that $k$ be algebraically closed from Proposition \ref{algclosed}.

\begin{proposition}\label{genlcase}
If $M = MV \cap F$ and $ord_R(M) = n$, then
$I(M) = \m^n$.
\end{proposition}
\begin{proof} 
Let $(\tilde{R},\tilde{\m})$ be a two-dimensional regular local ring with algebraically
closed residue field that is a faithfully flat overring of $R$ with
$\m \tilde{R} = \n$. Set $\tilde{M} = M \otimes_R \tilde{R}$ and denote by $\tilde{V}$ the order valuation
ring of $\tilde{R}$.

Observe that $\tilde{M} \subseteq \tilde{F} = F \otimes_R \tilde{R}$ (by flatness) and so is a torsion
free $\tilde{R}$-module with the quotient $\tilde{F}/\tilde{M}$ of finite
length (equal to the length of $F/M$). Also observe that a representing matrix
for $M$ over $R$ also represents $\tilde{M}$ over $\tilde{R}$ (when its entries are regarded
as coming from $\tilde{R}$). In particular, $I(\tilde{M}) = I(M)\tilde{R}$ and so by faithful flatness
$I(M) = I(\tilde{M}) \cap R$. Thus to show that $I(M)$ is a power of $\m$ it suffices
to see that $I(\tilde{M})$ is a power of $\tilde{\m}$.

This is done by appealing to Proposition \ref{algclosed}. The only thing that
needs verification is that $\tilde{M}$ is contracted from $\tilde{V}$. Choose
$C \subseteq M$ as in Lemma \ref{Clem}. 
By Lemma \ref{compute} we have $(\m^nC:_F det(C)) = M$. Let $\tilde{C}$ denote the
$\tilde{R}$-submodule of $\tilde{F}$ generated by $C$. We claim that $(\tilde{\m}^n\tilde{C}:_{\tilde{F}} det(\tilde{C})) = \tilde{M}$.
Given this, it follows from Lemma \ref{compute} that $\tilde{M}$ is
contracted from $\tilde{V}$, as desired.

Note that $det(\tilde{C}) = det(C)$ (they are both represented by the
same matrix with entries regarded in $\tilde{R}$ and $R$ respectively).
Also note that if $\m^nC$ is generated by $f_1,\cdots,f_k \in F$, then,
the $\tilde{R}$-module generated by these (regarded as elements
of $\tilde{F}$) is exactly $\tilde{\m}^n\tilde{C}$.

To show that $\tilde{M} \subseteq (\tilde{\m}^n\tilde{C}:_{\tilde{F}}det(\tilde{C}))$,
note that $det(C)M \subseteq \m^nC$. Since the $\tilde{R}$-modules generated by
$M$ and $\m^nC$ are exactly $\tilde{M}$ and $\tilde{\m}^n\tilde{C}$ (and $det(C) = det(\tilde{C})$), the desired
containment is clear.

To show the opposite containment, consider the sets
\begin{eqnarray*}
S &=& \{(f,z_1,\cdots,z_k) \in F \oplus R^k : det(C)f = z_1f_1 + \cdots + z_kf_k\}, {\text {~~and}}\\
\tilde{S} &=& \{(\tilde{f},\tilde{z}_1,\cdots,\tilde{z}_k) \in \tilde{F} \oplus \tilde{R}^k : det(C)\tilde{f} = \tilde{z}_1f_1 + \cdots + \tilde{z}_kf_k\}.
\end{eqnarray*}
Choosing a basis of $F$ identifies elements
of $F$ with elements of $R^r$ (where $r = rk_R(F)$). Suppose that $f$ is identified with the column vector $[x_1 x_2 \cdots x_r]^T$ and the $f_j$ with the column vectors $[f_{1j} f_{2j} \cdots f_{rj}]^T$.
Thus $S$ (respectively $\tilde{S}$) is the solution set in $R^{r+k}$ (respectively
$\tilde{R}^{r+k}$) of the set of
homogeneous linear equations
$$
det(C)\left[ \begin{array}{c}
              x_1 \\ x_2 \\ \vdots \\ x_r
              \end{array} \right]
=
\left[ 
\begin{array}{cccc}
f_{11} & f_{12} & \cdots & f_{1k} \\
f_{21} & f_{22} & \cdots & f_{2k}  \\
\vdots & \vdots & \ddots & \vdots \\
f_{r1} & f_{r2} & \cdots & f_{rk}
\end{array}\right]
\left[ \begin{array}{c}
              z_1 \\ z_2 \\ \vdots \\ \vdots \\ z_k
              \end{array} \right]
$$
in the variables $x_1,\cdots,x_r$ and $z_1,\cdots,z_k$ and coefficients given by $det(C)$ and the $f_{ij}$.
The equational criterion for flatness - see Theorem 7.6 of \cite{Mts1986} - 
now implies that any element of $\tilde{S}$ is an $\tilde{R}$-linear combination
of elements of $S$. In particular, projecting an element of $\tilde{S}$ onto
its $\tilde{F}$ part gives an $\tilde{R}$-linear combination of projections
of elements of $S$ onto their $F$ parts. Thus any element of $(\tilde{\m}^n\tilde{C}:_{\tilde{F}} det(\tilde{C}))$ is an $\tilde{R}$-linear combination
of elements of $(\m^nC :_F det(C)) = M$. Since $\tilde{M}$ is exactly the
set of $\tilde{R}$-combinations of elements of $M$, we have established the other
containment and finished the proof.
\end{proof}

\begin{remark}
A natural question that arises from the proof of Proposition \ref{algclosed}
is whether an ideal that is contracted from all $S = R[\frac{\m}{x}]$ is necessarily
a power of $\m$ (even without the residue field being algebraically closed). We give
an example to show that this need not be the case. Consider, for instance, the ideal
$I = (x^3,x^2y,x^2+y^2)$ in the ring $R = {\mathbb R}[| x,y |]$. It is easy to
check that $(I:x) = \m^2 = (I:\m)$ and so $I$ is contracted from $S =  R[\frac{\m}{x}]$.
Further $IS = x^2(x,1+\frac{y^2}{x^2})S$ which is the product of a principal ideal
and the maximal ideal $(x,1+\frac{y^2}{x^2})S$ and therefore integrally closed. It follows
that $I$ itself is integrally closed.

Further we claim that for any $z  \in \m \backslash \m^2$, we have that $(I:z) = (I:\m)$ ( = $\m^2$), so that $I$ is contracted from $S = R[\frac{\m}{z}]$.  To prove this claim, write
$z = \sum_{n \geq 1} p_n(x,y)$ where $p_n$
is a homogeneous polynomial of degree $n$
and $p_1(x,y) \neq 0$.
Suppose that $u \in (I:z)$. Note that $u \in \m$ and hence we may write $u = \sum_{n \geq 1} q_n(x,y)$ where $q_n$
is a homogeneous polynomial of degree $n$.
It suffices to see that $u \in \m^2$, or equivalently, that $q_1(x,y) = 0$.

Observe that the degree 2 component of any
element of $I$ is a scalar multiple of
$x^2+y^2$ while the degree 2 component of
$uz$ is $q_1(x,y)p_1(x,y)$. Since $x^2+y^2$
does not factor into linear polynomials over
${\mathbb R}$, it follows that $q_1(x,y) = 0$, as desired.
\end{remark}

\begin{corollary}\label{onebase}
If $M = MV \cap F$, then for any quadratic transform $T$ of $R$ other than $R$ itself,
$MT$ is a free $T$-module.
\end{corollary}
\begin{proof} 
By Proposition \ref{genlcase}, it follows that $I(M) = \m^n$ with $n = ord_R(M)$.
Thus any quadratic transform $T$ of $R$ other than $R$ itself, $I_{FT}(MT) = \m^nT$ is principal
and so by Lemma \ref{prinfree}, $MT$ is a free $T$-module.
\end{proof}

We now have a useful characterisation of
modules contracted from the order valuation
ring.

\begin{theorem}\label{charcont}
The following conditions are equivalent 
for $M$:
\begin{itemize}
\item[(1)] $M = MV \cap F.$
\item[(2)] $M$ is contracted from every $S = R[\frac{\m}{x}]$ and $I(M)$ is
a power of $\m$.
\item[(3)] $M$ is contracted from some $S = R[\frac{\m}{x}]$ and $I(M)$ is
a power of $\m$.
\item[(4)] $M$ is integrally closed and $I(M)$ is
a power of $\m$.
\end{itemize}
\end{theorem}
\begin{proof} (1) $\Rightarrow$ (2) holds since each $S = R[\frac{\m}{x}] \subseteq V,$ and by an application of Proposition \ref{genlcase} while (2) $\Rightarrow$ (3) is obvious. To show that (3) $\Rightarrow (1)$, observe first that for all first quadratic transforms $T$ of $R$, the $T$-module $MT$
is free since $I_{FT}(MT) = I(M)T$ is principal. Now apply Proposition \ref{contquad} to
conclude that $M = MV \cap F.$
To see that (1) $\Rightarrow (4)$ observe
that the valuative criterion of Theorem \ref{basic} implies that $M \subseteq \overline{M} 
= F \cap (\cap_V MV)$ over all valuation rings $V$
of $K$ containing $R$. It follows that if
$M = MV \cap F$ for any one valuation ring
$V \supseteq R$, then $M$ is integrally closed. The other part of (4) follows from
Proposition \ref{genlcase}. Finally (4) $\Rightarrow$ (3) is clear from Theorem \ref{Kdymain}.
\end{proof}

We will have occasion in the sequel to use the following corollary of Theorem~\ref{charcont}.

\begin{corollary}\label{stable}
For any $M$ and any $n \geq 1$,
$S_n(MV \cap F) = S_n(M)V \cap S_n(F)$.
\end{corollary}

\begin{proof}
Begin by noting that the characterisation
of the double dual implies that $S_n(F)$
is the double dual of $S_n(M)$.
First suppose that $M$ is contracted from $V$, so that $M = MV \cap F$. In this case,
we need to see that so are all $S_n(M)$, for
$n \geq 1$.
By Theorem \ref{charcont}, $M$
is integrally closed (with $I(M)$ a power of $\m$) and then by Theorem \ref{Kdymain} all $S_n(M)$ are integrally
closed. Further it is easy to see that
$I(S_n(M))$ is, in general, a power of $I(M)$
(with exponent given by an appropriate binomial coefficient) and so, in this case, is a power of $\m$. By Theorem \ref{charcont} again, $S_n(M)$ is contracted
from $V$.

For a general $M$, it is clear that 
$S_n(MV \cap F) \subseteq S_n(M)V \cap S_n(F)$.
For the opposite inclusion, use that $S_n(MV \cap F)$ is contracted from $V$ (by the previous paragraph) and contains $S_n(M)$
(obviously) to conclude that it contains
$S_n(M)V \cap S_n(F)$.
\end{proof}

\subsection{The analogue of the Hoskin-Deligne formula}

We recall that the Hoskin-Deligne formula for integrally closed ideals in $R$ asserts
that if $I$ is an integrally closed $\m$-primary ideal of $R$, then
$$
\lambda_R(\frac{R}{I}) = \sum_{T \geq R} \binom{ord_T(I^T)+1}{2} [T:R],
$$
where the sum is over all quadratic transforms $T$ of $R$, $ord_T(\cdot)$ denotes
the order in the discrete valuation ring $V_T$ associated to the maximal ideal of $T$, and
$[T:R]$ is the residue field extension degree.

We begin the proof of the analogue of the Hoskin-Deligne length formula for integrally
closed modules by proving a preparatory lemma for the basic inductive step.

\begin{lemma}\label{ordlemma} $ord_R(M) = \lambda_V(FV/MV)$.
\end{lemma}

\begin{proof}
By definition, $ord_R(M) = ord_R(I(M))$. To compute this, we may extend $I(M)$
to $V$ and take its colength. Note that 
$MV \subseteq FV$ are free modules of equal rank ($=rk_R(M)$) over the 
PID $V$ and so 
the
extension of $I(M)$ to $V$ may also be described as the ideal generated by the
determinant of the matrix representing 
a basis of $MV$ written in terms of a basis of $FV$.
Now an appeal to the structure of modules over a principal ideal domain shows that the colength
of the ideal generated by the determinant is the length of $FV/MV$, as
desired.
\end{proof}

Next we prove the inductive step which holds more generally for contracted
modules. One of the observations that we will use in its proof is that
if $S = R[\frac{\m}{x}]$ for $x \in \m \backslash \m^2$, then, $$MS = \bigcup_{n \geq 0} \frac{\m^n}{x^n} M,$$
regarded as subsets of $M_K$.

\begin{proposition}\label{isom} If $M$ is contracted from $S = R[\frac{\m}{x}]$ for $x \in \m \backslash \m^2$, then the natural map
\begin{equation}\label{natural} \frac{MV \cap F}{M} \rightarrow \frac{(MV \cap F)S}{MS}
\end{equation}
is an isomorphism (of $R$-modules).
\end{proposition}
\begin{proof} 
Begin by observing that $M \subseteq MV \cap F \subseteq F$ and so
the characterisation of the double dual from \S \ref{Kdysum} imples that $F$ is also the
double dual of $MV \cap F$.
Next, observe by an  application 	of
Lemma \ref{ordlemma} that $ord_R(M) = ord_R(MV \cap F)$ since both of these
extend to the same submodule, namely $MV$, of $FV$.

Now, since $MV \cap F$ is contracted from $V$, it is certainly also
contracted from its subring $S$ (it is easily seen that $V$ is the localisation
of $S$ at its height 1 prime $\m S$). We now apply Proposition \ref{contracted}
(to both $M$ and $MV \cap F$) and the above equality of orders to conclude that
$$
\lambda_R(F/MV \cap F +xF) = ord_R(MV \cap F) = ord_R(M) = \lambda_R(F/M+xF),
$$
or equivalently that $MV \cap F + xF = M + xF$.
In particular, $MV \cap F \subseteq M+xF$, and therefore
\begin{equation}\label{mvvsm}
MV\cap F = M + x(MV\cap F:_F x) = M + x(MV\cap F :_F \m).
\end{equation}
where the last equality follows from Proposition \ref{contracted} applied
to $MV \cap F$.

We will now prove by induction that 
\begin{equation}\label{fund}
\m^n(MV \cap F) = \m^nM + x^n(MV \cap F),
\end{equation} 
for all $n \geq 0$.
The case $n=0$ is trivial and the basis case $n=1$ follows easily from Equation (\ref{mvvsm}).
As for the inductive step, we have that for $n \geq 1$,
\begin{eqnarray*}
\m^{n+1}(MV \cap F) &=&  \m^{n+1}M + x^n\m(MV \cap F)\\
                    &=&  \m^{n+1}M + x^n(\m M + x(MV \cap F))\\
                    &=&  \m^{n+1}M + x^{n+1}(MV \cap F),
\end{eqnarray*}
where the first equality follows from the inductive assumption and the second from
the basis case.

Equation (\ref{fund})  implies that
$$\bigcup_{n \geq 0}\frac{\m^n}{x^n}(MV \cap F) = \bigcup_{n \geq 0}\frac{\m^n}{x^n}M
+ (MV \cap F),$$
and so in view of the observation preceding the statement of this proposition,
\begin{equation*}\label{surjects}
(MV \cap F)S = MS + (MV \cap F).
\end{equation*}
Hence the natural map of Equation (\ref{natural}) is surjective and its injectivity 
follows easily since $M$ is contracted.
\end{proof}

The next proposition restates Proposition \ref{isom} in terms
of quadratic transforms.

\begin{proposition}\label{contquad} If $M$ is contracted
from $S = R[\frac{\m}{x}]$ for $x \in \m \backslash \m^2,$ then,
$$
\frac{MV \cap F}{M} \cong \bigoplus_T \frac{(MT)^{**}}{MT}$$
where the direct sum extends over all first quadratic transforms $T$ of 
$R$ (and the corresponding summand vanishes except for finitely many $T$).
\end{proposition}

\begin{proof} 
In view of Proposition \ref{isom} it is to be seen that
$$\frac{(MV \cap F)S}{MS} \cong \bigoplus_T \frac{(MT)^{**}}{MT}$$
where the direct sum extends over all first quadratic transforms $T$ of
$R$.

Observe that the module $N = \frac{(MV \cap F)S}{MS}$ is an $S$-module that is
of finite length as an $R$-module and hence also an $S$-module. Thus its
support is a set of finitely many maximal ideals, say $Q_1,\cdots,Q_k$, of $S$.
Each of these maximal ideals necessarily contains $\m S$, since any
other maximal ideal of $S$ contracts to a height 1 prime of $R$,
at which the localisation of $N$ vanishes. Hence $N$ is isomorphic
to the direct sum of its localisations at these maximal ideals.

Let $Q$ be one such maximal ideal of $S$ in the support of $N$. By definition,
$T = S_Q$ is a first quadratic transform of $R$. Then, 
$$
N_Q \cong \frac{(MV \cap F)T}{MT},
$$
which is a $T$-module of finite length (being a summand of $N$). Since
$(MV \cap F)T$ is a free $T$-module by Corollary \ref{onebase}, the characterisation
of the double dual shows that $(MV \cap F)T = (MT)^{**}$.

We see therefore that $N$ is isomorphic to a direct sum of all $(MT)^{**}/MT$ where the direct sum ranges over
all (first) quadratic
transforms $T$ of $R$ that are localisations of $S$.
It only remains to see that if $T$ is the first quadratic transform of $R$ that
is not a localisation of $S$, then $MT$ is a free $T$-module. Thus $T = R[\frac{\m}{y}]_{(y,\frac{x}{y})}$ where
$\m = (x,y)R$. Since $M$ is contracted from $S$, it follows from Proposition
\ref{contracted} that so is $I(M)$ and hence from Equation (\ref{mvvsm}) (applied
to $I(M)$) that 
$$
\m^n = I(M) + x(I(M):\m) = I(M) + x\m^{n-1} \Rightarrow y^n \in I(M) + x\m^{n-1},
$$
where $n = ord_R(I(M))$ ($= ord_R(M)$). Divide by $y^n$ and read the equation
in $T$ to conclude that $1 \in I(M)^T + \m_T$ (where $\m_T$ is the maximal ideal
of $T$) and therefore that $I(M)^T = T$. Thus $I_{FT}(MT) = I(M)T$ is principal (generated by
$y^n$) and so $MT$ is free by Lemma \ref{prinfree}.
\end{proof}

We now prove the following analogue of the Hoskin-Deligne length formula for integrally closed modules.

\begin{theorem}\label{hdm} Let $R$ be a two-dimensional regular local ring with maximal ideal $\m$
and infinite residue field.
Let $M$ be a finitely generated, torsion-free, integrally closed $R$-module.
Then,
\begin{equation}\label{hdf}
\lambda_R(\frac{F}{M}) = \sum_T \lambda_T(\frac{(MT)^{**}}{MV_T \cap 
(MT)^{**}})[T:R]
\end{equation}
where the sum extends over all quadratic transforms $T$ of $R$. 
\end{theorem}

\begin{proof} The proof proceeds by induction on 
the length of $F/M$.
If $\lambda_R(F/M) = 0$, then $M = F$ is a free $R$-module. 
Hence, for each quadratic transform $T$ of $R$, the module $MT$ is a free $T$-module and therefore all terms on the right 
in equation (\ref{hdf}) also vanish.

Next, we suppose that $\lambda_R(F/M) \geq 1$. 
We may write
\begin{equation}\label{first}
\lambda_R(\frac{F}{M}) = \lambda_R(\frac{F}{MV \cap F}) + 
\lambda_R(\frac{MV \cap F}{M}).
\end{equation}
Note that $MV \cap F \neq F$ since otherwise $MV = FV$ and so by
Lemma \ref{ordlemma}, $ord_R(M) = 0$. This implies that $I(M) =R$
and so $M$ would be free (and therefore equal to $F$). 

Since $M$ is integrally closed,
it is contracted from
$S = R[\frac{\m}{x}]$ for some $x \in \m \backslash \m^2$ by Proposition \ref{contracted}. Now applying Proposition \ref{contquad} to the last term
in Equation (\ref{first}) gives:
\begin{equation}\label{second}
\lambda_R(\frac{F}{M}) = \lambda_R(\frac{F}{MV \cap F}) +
\sum_T \lambda_T(\frac{(MT)^{**}}{MT})[T:R]
\end{equation}
where the sum extends over all first quadratic transforms $T$ of $R$.
The multiplicative factor $[T:R]$ arises since any finite length $T$-module
is also a finite length $R$-module and its lengths as $R$- and $T$-modules
differ by exactly this factor.

Since $MV \cap F \neq F$, Equation (\ref{second}) shows that $\lambda_T(\frac{(MT)^{**}}{MT}) <
\lambda_R(\frac{F}{M})$ for any $T$ occurring in the sum.
Also, by Theorem \ref{Kdymain}, $MT$ is an integrally closed module and 
so by induction we may assume that 
 for any such $T$ we have
\begin{equation}\label{third}
\lambda_T(\frac{(MT)^{**}}{MT}) = \sum \lambda_{T^\prime}(
\frac{(MT^\prime)^{**}}{MV_{T^\prime} \cap (MT^\prime)^{**}})
[T^\prime : T].
\end{equation}
where the sum extends over all quadratic transforms $T^\prime$ of $T$.

Finally, substituting the expression from Equation (\ref{third}) into
Equation (\ref{second}) and using the facts - see \S \ref{Kdysum} - that (i) any quadratic
transform of $R$ is either $R$ itself or a quadratic transform of a
unique first quadratic transform of $R$ and (ii)
multiplicativity
of the residue field extension degree, we have the desired result.
\end{proof}

\begin{remark} We note the sense in which Theorem \ref{hdm} is an analogue
of the Hoskin-Deligne formula for ideals. The terms on the right in Equation
(\ref{hdf}) depend only on the modules $MV_T \cap (MT)^{**}$. These are all
modules that are contracted from the order valuation rings of various quadratic
transforms of $R$. Thus the colength (in its double dual) of an 
integrally closed module is expressed in terms of the colengths of modules
contracted from the order valuation rings of quadratic
transforms of $R$. When $M=I$ - an $\m$-primary ideal of $R$, $(MT)^{**}$
is just $\m_R^{ord_T(IT)}T$ and so we recover the ideal form of the result.
\end{remark}

\begin{remark} Observe that our proof of Theorem \ref{hdm} shows that it
is more than just a numerical result. What is actually seen is that the finite length $R$-module $F/M$
has a filtration where the successive quotients are exactly those of the form $(MT)^{**}/(MV_T \cap (MT)^{**})$
where $T$ is a quadratic transform of $R$.
\end{remark}

An immediate corollary of the analogue of the Hoskin-Deligne formula is
an expression for the Buchsbaum-Rim multiplicity in terms of those
of modules contracted from the order valuation rings of various quadratic
transforms of $R$.

\begin{corollary}\label{brcor} If $M$ is an integrally closed
$R$-module, then,
$$e_R(M) = \sum e_T(MV_T \cap (MT)^{**})[T : R]$$
where the sum extends over all quadratic transforms $T$ of $R$.
\end{corollary}
\begin{proof} By Theorem \ref{Kdymain}, for each $n \geq 1$, $S_n(M)$ is
integrally closed and by applying Theorem \ref{hdm} to $S_n(M)$ we get:
\begin{eqnarray*}
\lambda_R\left(\frac{S_n(F)}{S_n(M)}\right) &=& \sum \lambda_T\left(\frac{(S_n(M)T)^{**}  
}{S_n(M)V_T \cap 
(S_n(M)T)^{**}}\right)[T:R]\\
&=& \sum \lambda_T\left(\frac{(S_n(MT))^{**}  
}{S_n(MV_T \cap (MT)^{**})}\right)[T:R],
\end{eqnarray*}
where the sum extends over all quadratic transforms $T$ of $R$ and the second equality follows from Corollary \ref {stable} applied to $MT$.

Only finitely many $T$ contribute to the
sum on the right and these are those for
which $MV_T \cap (MT)^{**} \neq (MT)^{**}$.
For any such $T$, the term corresponding to
$T$ is given by a polynomial in $n$ of degree $1+rk_T(MT) = 1+rk_R(M)$ and
leading coefficient $e_T(MV_T \cap (MT)^{**})$, if $n$ is sufficiently large.
The desired equality follows.
\end{proof}

\section{A multiplicity formula for integrally closed modules}

The main result in this section - Corollary \ref{multformula} - establishes a relation
between the colengths and multiplicities of an integrally closed module and
its ideal of minors. While the result itself refers only to a single two-dimensional regular local ring, the proof seems to need the machinery of
quadratic transforms, exactly as in the proof of, say, Zariski's product theorem. Also, note that this result is genuinely a module statement - the
ideal case of this being a triviality.

\begin{proposition}\label{lmmodn} If $M$ is
integrally closed with $I(M) = \m^n$ and
 $N$ is a minimal reduction of $M$,  then 
$\lambda_R(\frac{M}{N}) = \binom n2.$
\end{proposition}
\begin{proof}

We first reduce to the case
that $M \subseteq \m F$. Write $M = G \oplus \tilde{M}$ where $G$ is free and $\tilde{M}$ has no free direct summand. Taking double
duals gives $F = G \oplus \tilde{F}$, with the obvious notation. Now $\tilde{M} \subseteq \m \tilde{F}$ and is integrally closed. Further $I(\tilde{M}) = I(M) = \m^n$ and if $\tilde{N}$ is
a minimal reduction of $\tilde{M}$, then $G \oplus \tilde{N}$ is a minimal reduction
of $M$. Thus it would follow that
$$
\lambda_R(\frac{M}{N}) = \lambda_R(\frac{\tilde{M}}{\tilde{N}}) = \binom{n}{2},
$$
by the reduction. Hence we assume that $M \subseteq \m F$. Thus $n \geq rk_R(M) = r$, say. Note that $\nu_R(M) = r + n$.

The proof proceeds by induction on $n$. 
If $n=0$, then $r=0$, so $M=0$ and the result
is clear. If $n=1$ and $M \neq 0$, then $r=1$
and so $M = N = \m$ and here too the result
is easily verified.

Now suppose that $n > 1$. Extend a minimal generating set of $N$ (with $r+1$ elements)
to one of $M$ (with $r+n$ elements). Use these generating sets to construct minimal resolutions of $N$ and $M$ and extend the
inclusion map $N \hookrightarrow M$ to the
following map of complexes:
$$
\CD
0 @>>> R @> \Delta 
>>
R^{r+1} @>>> N @>>> 0\\
@. @V{X 
}
VV @V {\left[ \begin{matrix}
I \\ 0 \end{matrix} \right]}VV @V VV \\
0 @>>> R^n @>A>> R^{n+r} @>>> M @>>> 0.
\endCD
$$
Here $A$ is a $(n+r) \times n$ presenting matrix of $M$, $I$ is
the identity matrix of size $r+1$,  $0$ is the zero matrix of size $(n-1) \times (r+1)$ 
and $\Delta$ presents $N$ (so that the entries of $\Delta$ generate $I(N)$).

Standard homological algebra implies that
the mapping cone of the map of complexes
above resolves $M/N$. Explicitly, the following complex is exact:

$$
\CD
0 @>>>  R @>{\left( \begin{matrix} X \\ -{\Delta}\end{matrix} \right)}>> R^n \oplus R^{r+1} @>{\left( \begin{matrix}
B & I\\
C & 0
\end{matrix} \right)}>> R^{n+r} @>>> \frac{M}{N} @>>> 0,
\endCD
$$ 
where $B$ is the matrix formed by the first
$r+1$ rows of $A$ and $C$ is matrix formed
by the remaining rows of $A$.

We split off free direct summands to get the 
following minimal resolution:
$$
\CD
0 @>>> R @>
X>>
R^n @>
C>>
R^{n-1} @>>> \frac{M}{N} @>>> 0.
\endCD
$$
Let $P$ be the image of $C$, which is
a submodule of $\m R^{n-1}$ (since all
entries of $C$ are in $\m$) with
finite length quotient isomorphic to
$\frac{M}{N}$, and let $\overline{P}$ be
its integral closure (which is also a
submodule of $\m R^{n-1}$, since $\m R^{n-1}$ is integrally closed).
It follows from the above resolution that $I(P)$ is generated by the entries of $X$.

We claim that $P$ is a minimal reduction
of $\overline{P}$ and that $I(\overline{P}) = \m^{n-1}$. To see this, note that $rk_R(P)
= n-1$ while $\nu_R(P) = n = rk_R(P) + 1$
and so $P$ is indeed a minimal reduction of
$\overline{P}$. Next we have by Theorems \ref{basic} and \ref{Kdymain} that $I(\overline{P}) = \overline{I(P)}$.
Thus we need to see that $I(P)$ is a reduction of $\m^{n-1}$. Clearly, $I(P) \subseteq \m^{n-1}$ since it is the ideal
of maximal minors of $C$ which has all entries in $\m$. Also the map of complexes
above shows that $AX = \left[ \begin{array}{l} \Delta \\ 0 \end{array} \right]$ and
therefore that $I(N) \subseteq \m I(P)$ (since the entries of $\Delta$ generate $I(N)$ and those of $X$ generate $I(P)$).
Taking integral closures (and using Zariski's product theorem) we get $\m^n = 
I(M) = \overline{I(N)} \subseteq \m I(\overline{P})$. By the determinant trick,
$\m^{n-1} \subseteq \overline{I(\overline{P})} = \overline{I(P)}$. Thus, $I(P)$ is a reduction of $\m^{n-1}$ establishing the claim.
It now follows by induction on $n$ that
$\lambda_R(\frac{\overline{P}}{P}) = \binom{n-1}{2}$.

Next, we claim that $\overline{P} = \m R^{n-1}$. 
We already know that $\overline{P} \subseteq \m R^{n-1}$.
To see the opposite inclusion, first apply Theorem 
\ref{charcont} to see that $\overline{P}$
is contracted from $V$ and then by Lemmas
\ref{Clem} and \ref{compute} find $C \subseteq \overline{P}$
so that $(\m^{n-1}C :_{R^{n-1}} det(C)) = CV \cap R^{n-1}  = \overline{P}$.
Now note that 
$det(C)R^{n-1} \subseteq \m^{n-2}C$ since
$det(C) = C adj(C)$ and all entries of $
adj(C)$ are in $\m^{n-2}$, and therefore
$det(C)\m R^{n-1} \subseteq \m^{n-1}C$.
So $\m R^{n-1} \subseteq \overline{P}$,
as needed. Thus $\lambda_R(\frac{R^{n-1}}{\overline{P}}) = \lambda_R(\frac{R^{n-1}}{\m R^{n-1}}) = n-1$.

Finally, we have
$$ \lambda_R(\frac{M}{N}) = \lambda_R(\frac{R^{n-1}}{P}) = 
\lambda_R(\frac{R^{n-1}}{\overline{P}}) + \lambda_R(\frac{\overline{P}}{P}) = (n-1) + \binom {n-1}2 = \binom n2.$$
\end{proof}

An immediate corollary of Proposition \ref{lmmodn} is the following `local' version
of the multiplicity formula. Here, we use `local' in the sense of `valid for
modules contracted from the order valuation ring'.

\begin{corollary}\label{local} If $M = MV \cap F$ and $n = \ord (M)$,
then,
$$e_R(M) = \lambda_R(\frac{F}{M}) + \binom n2.$$
\end{corollary}
\begin{proof} If $N \subseteq M$ is a minimal reduction of $M$, then,
by Proposition \ref{amalgam}, $e_R(M) = \lambda_R(F/N)$. Now appeal to Proposition
\ref{lmmodn}.
\end{proof}

As an application of the analogue of the Hoskin-Deligne length formula for integrally closed modules, we now derive an interesting numerical
relationship between the multiplicities and colengths of an integrally
closed module $M$ and its ideal of minors $I(M)$.

\begin{corollary}\label{multformula} If $M$ is  integrally closed with ideal of minors $I(M)$,
then,
$$e_R(M) = e_R(I(M)) - \lambda_R(\frac{R}{I(M)}) + \lambda_R(\frac{F}{M}).$$
\end{corollary}

\begin{proof} Note that by Theorem \ref{Kdymain}, $I(M)$ is integrally closed. We have expressions for each term above in terms of modules
contracted from order valuations of various quadratic transforms of $R$. Explicitly, we have by Corollary \ref{brcor} applied to $M$ that,
$$
e_R(M) = \sum_T e_T(MV_T \cap (MT)^{**})[T:R],
$$
and by the same corollary applied to $I(M)$ that,
\begin{eqnarray*}
e_R(I(M)) &=& \sum_T e_T\left(I(M)V_T \cap (I(M)T)^{**}\right)[T:R]\\ &=& \sum_T e_T(\m_T^{ord_T(I(M)^T)})[T:R]\\ &=& \sum_T ord_T(I(M)^T)^2[T:R].
\end{eqnarray*}
Similarly, by Theorem \ref{hdm} applied to $M$ we have
\begin{equation*}
\lambda_R(\frac{F}{M}) = \sum_T \lambda_T(\frac{(MT)^{**}}{MV_T \cap 
(MT)^{**}})[T:R],
\end{equation*}
and by the same theorem applied to $I(M)$ that
\begin{eqnarray*}
\lambda_R(\frac{R}{I(M)}) &=& \sum_T \lambda_T\left(\frac{(I(M)T)^{**}}{I(M)V_T \cap 
(I(M)T)^{**}}\right)[T:R] \\ &=& \sum_T \lambda_T\left(\frac{T}{\m_T^{ord_T(I(M)^T)}}\right)[T:R]\\ &=& \sum_T \binom{ord_T(I(M)^T)+1}{2}[T:R].
\end{eqnarray*}

To finish the proof it therefore suffices to see that
\begin{eqnarray*}
e_T(MV_T \cap (MT)^{**}) &=& ord_T(I(M)^T)^2 - \binom{ord_T(I(M)^T)+1}{2} + \lambda_T\left(\frac{(MT)^{**}}{MV_T \cap 
(MT)^{**}}\right)\\
&=& \binom{ord_T(I(M)^T)}{2} + \lambda_T\left(\frac{(MT)^{**}}{MV_T \cap 
(MT)^{**}}\right)
\end{eqnarray*}
A little thought now shows that this equality is exactly the content
of Corollary \ref{local} applied to the $T$-module $MV_T \cap 
(MT)^{**}$.
\end{proof}

\end{document}